\documentclass[a4paper,11pt,reqno]{amsart}
\usepackage{a4wide}
\usepackage{enumerate}
\usepackage{amssymb, amsmath}
\usepackage{mathrsfs}
\usepackage{amscd}
\usepackage[active]{srcltx}
\usepackage{fancyvrb,color}
\usepackage{verbatim}
\usepackage[colorlinks,linkcolor={black},citecolor={blue},urlcolor={black}]{hyperref}
\definecolor{darkgreen}{rgb}{0.00,0.33,0.25}
\definecolor{darkred}{rgb}{0.60,0.05,0.05}
\definecolor{darkblue}{rgb}{0.05,0.05,0.60}

\theoremstyle{plain}
\newtheorem{theorem}{Theorem}[section]
\newtheorem{corollary}[theorem]{Corollary}
\newtheorem{lemma}[theorem]{Lemma}
\newtheorem{proposition}[theorem]{Proposition}

\newtheorem{definition}[theorem]{Definition}

\newtheorem*{definition*}{Definition}

\theoremstyle{remark}
\newtheorem{remark}[theorem]{Remark}

\newtheorem*{claim*}{Claim}
\newtheorem*{remark*}{Remark}
\newtheorem*{example*}{Example}
\newtheorem*{notation*}{Notation}

\numberwithin{equation}{section}

\def\N{{\mathbb N}}
\def\Z{{\mathbb Z}}
\def\R{{\mathbb R}}

\newcommand{\G}{\mathcal G}
\newcommand{\one}{{{\bf 1}}}

\newcommand{\eps}{\varepsilon}
\renewcommand{\phi}{\varphi}

\newcommand{\dd}{\; \mathrm{d}}

\newcommand{\norm}[1]{\| {#1}\|}
\newcommand{\abs}[1]{\vert {#1}\vert}

\DeclareMathOperator{\Ric}{Ric}

\DeclareMathOperator{\Hess}{Hess}

\newcommand{\ddt}{\frac{\mathrm{d}}{\mathrm{d}t}}

\newcommand{\cH}{\mathcal{H}}
\newcommand{\cB}{\mathcal{B}}

\newcommand{\cW}{\mathcal{W}}

\newcommand{\cX}{\mathcal{X}}
\newcommand{\cE}{\mathcal{E}}
\newcommand{\cA}{\mathcal{A}}

\newcommand{\cP}{\mathscr{P}}
\newcommand{\PX}{\cP(\cX)}
\newcommand{\PXs}{\cP_*(\cX)}
\newcommand{\hrho}{\hat\rho}

\renewcommand{\tilde}{\widetilde}

\begin{document}

\title{Ricci curvature bounds \\for weakly interacting Markov chains}

\author{ {Matthias Erbar}, {Christopher Henderson}, {Georg Menz},
  {Prasad Tetali} } 

\address{Matthias Erbar,
Institute for Applied Mathematics\\
University of Bonn\\
Endenicher Allee 60\\
53115 Bonn\\
Germany} 
\email{erbar@iam.uni-bonn.de}

\address{Christopher Henderson, Labex MILYON \& UMPA \& INRIA Numed Team\\ \'Ecole normale superieure de Lyon\\
46 All\'ee d'Italie\\
69007 Lyon\\ France} 
\email{christopher.henderson@ens-lyon.fr}

\address{Georg Menz, Mathematics Department\\
UCLA\\ 
Los Angeles\\ CA 90095\\ USA} 
\email{gmenz@math.ucla.edu}

\address{Prasad Tetali,
School of Mathematics\\ 
Georgia Institute of Technology\\ 
Atlanta\\ GA 30332\\ USA}
\email{tetali@math.gatech.edu}



\maketitle

 \begin{abstract}
   We establish a general perturbative method to prove entropic Ricci
   curvature bounds for interacting stochastic particle systems. We
   apply this method to obtain curvature bounds in several examples,
   namely: Glauber dynamics for a class of spin systems including the
   Ising and Curie--Weiss models, a class of hard-core models and
   random walks on groups induced by a conjugacy invariant set of
   generators.
 \end{abstract}

\tableofcontents

\section{Introduction}
\label{sec:intro}

Bounds on the Ricci curvature are an essential ingredient to control
the behavior of diffusion processes on Riemannian manifolds. For
instance, the celebrated Bakry--\'Emery criterion asserts
that a bound $\Ric+\Hess V\geq \lambda>0$ guarantees that the drift
diffusion process with generator $Lu=\Delta u-\nabla V\cdot\nabla u$
satisfies a logarithmic Sobolev inequality. The latter controls the
trend to equilibrium of the associated semigroup through the
exponential decay of the entropy. Furthermore, a large number of other
geometric and functional inequalities can be derived from curvature
bounds.

In view of this wide range of implications, considerable effort has
been devoted to developing a notion of (lower bounds for the) Ricci
curvature for non-smooth spaces. Bakry and \'Emery \cite{BE85}
introduced an approach based on algebraic properties of diffusion
operators, the so-called $\Gamma_2$-calculus. A different approach
based on optimal transport has been taken by Sturm \cite{S06} and Lott
and Villani \cite{LV09} and applies to metric measure spaces. Such a
space is said to have Ricci curvature bounded below by $\kappa$,
provided the relative entropy is $\kappa$-convex along geodesics in
the Wasserstein space of probability measures. As in the smooth case
these notions of curvature bounds entail a large number of functional
inequalities.

Unfortunately, this theory does not apply to discrete spaces and
Markov chains and many alternative notions of Ricci curvature bounds
have been developed in this setting, see
e.g.~\cite{BS09,GRST13,Oll10}. We will focus on the notion of
\emph{entropic Ricci curvature bounds} put forward in \cite{EM12,Ma11}
which applies to a finite Markov chain and seems particularly well
suited to study functional inequalities in the discrete setting. Here
the idea is to replace the role of the $L^2$-Wasserstein distance with
a new transportation distance in the Lott--Sturm--Villani
definition. It has been shown in \cite{EM12} that, in analogy with the
Bakry--\'Emery criterion, a strictly positive entropic Ricci bound
implies a modified logarithmic Sobolev inequality (MLSI). Moreover, it
entails a Poincar\'e inequality and an analogue to Talagrand's
transport cost entropy inequality.

In view of these consequences, it is desirable to obtain entropic
Ricci bounds in concrete examples of Markov chains. Relatively few results
in this direction are available to date: Mielke derived entropic Ricci
bounds for one-dimensional birth and death chains and applied these to
discretizations of Fokker--Planck equations. Erbar--Maas \cite{EM12}
obtained a tensorization result giving an entropic Ricci curvature
bound for the product of two Markov chains in terms of Ricci bounds of
the individual chains. In particular, this allows to get sharp bounds
for the random walk on the hyper-cube $\{-1,1\}^n$. First results in
high dimensions beyond product chains were obtained by
Erbar--Maas--Tetali \cite{EMT15}, considering the simple exclusion
process on the complete graph and the random transposition shuffle
models. Fathi--Maas \cite{FM15} generalized the latter results by
considering inhomogeneous jump rates in these models and obtained new
results for the zero range process.

In this work, we present a general perturbative criterion to derive
entropic Ricci curvature bounds for weakly interacting Markov chains
and apply this method in a number of examples. Perturbation methods
are well-known in the study of functional inequalities, see for
instance the Holley-Stroock criterion for the logarithmic Sobolev
inequality (LSI). 

To formulate our main results, consider an irreducible and reversible
Markov chain on a finite set $\cX$ whose generator $L$ can be written
in the form
\begin{align*}
  L\psi (x) = \sum_{\delta\in G}\Big(\psi(\delta x) - \psi(x)\Big)c(x,\delta)\;,
\end{align*}
where $G$ is a collection of bijective maps $\delta:\cX \to \cX$ and
$c:\cX\times G\to\R_+$ are the transition rates. Let $\pi$ denote the
unique reversible probability measure on $\cX$, i.e.~$\pi$ satisfies
the detailed-balance condition $c(x,\delta)\pi(x)=c(\delta
x,\delta^{-1})\pi(\delta x)$ for all $x\in\cX,\delta\in G$. Then one
of our main results is the following (see Theorem~\ref{thm:perturb}
below).

\begin{theorem}\label{thm:main-intro}
  Assume that $\delta\eta x=\eta\delta x$ for all $x\in \cX,\delta,\eta\in G$ and that
  \begin{align}\label{eq:ass-rates-intro}
    \lambda~:=~\min_{\stackrel{x\in\cX,\delta\in
      G}{c(x,\delta)>0}}\left[c(x,\delta)-\one_{\delta\neq\delta^{-1}}c(\delta
      x,\delta) - \sum_{\eta:\eta\neq\delta,\delta^{-1}}
      \frac{(q-q_*)(\delta
        x,\delta^{-1},\eta)}{c(x,\delta)\pi(x)}\right]~\geq~0\;,
  \end{align}
  where we set $q(x,\delta,\eta)=c(x,\delta)c(x,\eta)\pi(x)$ as well
  as $q_*(x,\delta,\eta)=\min\{q(x,\delta,\eta),q(\delta
  x,\delta^{-1},\eta),$ $q(\eta x,\delta,\eta^{-1}),q(\delta \eta
  x,\delta^{-1},\eta^{-1})\}$. Then, the entropic Ricci curvature of
  the chain is bounded below by $2\lambda$.
\end{theorem}
 
That this is a perturbative criterion can be seen as follows. It is
typical of product situations that the jump rates are homogeneous, in
the sense that $c(\delta x,\eta)=c(x,\eta)$, for all
$x,\delta,\eta$. In this case, we find $\lambda\geq 0$ and recover the
criterion established in \cite{EM12}, used to prove the tensorization
principle for entropic Ricci bounds. Theorem~\ref{thm:main-intro} is a
generalization of this criterion when a {\em quantitative} bound on
the deficit in the homogeneity of the rates is given.  As a result, a
key advantage of our results is that it gives an explicit condition on
the transition rates that can be checked directly on examples.

We apply Theorem \ref{thm:main-intro} to derive new entropic Ricci
bounds for different statistical mechanics models. In particular, we
consider Glauber dynamics for the Ising model on a general weighted
graph and a general hard-core model. In the case of the hard-core model,
we recover, in particular, the criterion derived in \cite{DPP} for
convex decay of the entropy and the MLSI. In the Ising case, the maps
$\delta$ correspond to flipping individual spins. We show that
\eqref{eq:ass-rates-intro} is satisfied for sufficiently high
temperature. For the Ising model on square-lattice and the
Curie--Weiss model we obtain a positive bound on the Ricci curvature
that is uniform in the size of the system. We note that Ollivier
\cite[Ex.~17]{Oll09} has obtained a positive bound on his notion of
coarse Ricci curvature for this chain under weaker assumptions on
the temperature (in fact, down to the single-site Dobrushin
condition). However, this notion of curvature is not known to imply
the MLSI \eqref{eq:MLSI}, for instance, among other aspects.

Finally, we develop an analogue of Theorem \ref{thm:main-intro} for a
class of Markov chains based on non-commutative maps. Namely, we
consider random walks on Cayley graphs of non-abelian groups generated
by a set invariant under conjugation. Prototypical examples are random
walks on the symmetric group $S_n$ generated by $k$-cycles. Our result
also allows to treat {\em inhomogeneous} jump rates for the random walk.
For a precise formulation we refer to Theorem~\ref{thm:Cayley-Ricci}.

\medskip

\textbf{Organization:} In Section~\ref{sec:prelim}, we recall the basic
facts about entropic Ricci curvature bounds for finite Markov
chains. In Section~\ref{sec:perturb}, we introduce the new perturbative
approach to proving Ricci bounds and give the proof of the main
results. Finally, we apply this method to different examples in
Section~\ref{sec:examples}.

\subsection*{Acknowledgement}

The authors wish to thank Max Fathi, Jan Maas and Andr\'e Schlichting
for stimulating discussions on this work and related topics. This work
originated in discussions that took place at the SQuARE meetings
\emph{Displacement convexity for interacting Markov chains} at the
American Institute for Mathematics. The authors wish to thank AIM for
the inspiring athmosphere making this collaboration possible.  E.M
gratefully acknowledges support by the German Research Foundation
through the Collaborative Research Center 1060 \emph{The Mathematics
  of Emergent Effects} and the Hausdorff Center for Mathematics. Part
of this work was performed within the framework of the LABEX MILYON
(ANR- 10-LABX-0070) of Universit\'e de Lyon, within the program
“Investissements d’Avenir” (ANR-11- IDEX-0007) operated by the French
National Research Agency (ANR). G.M. gratefully acknowledges support
by the National Science Foundation under Grant No. DMS-1440140 while
he was in residence at the Mathematical Sciences Research Institute in
Berkeley, California, during the Fall 2015 semester. P.T. gratefully
acknwoledges support by the NSF grant DMS-1407657.

\section{Entropic Ricci curvature bounds for Markov chains}
\label{sec:prelim}

Here we briefly recall the definitions of the discrete transport
distance $\cW$, the entropic Ricci curvature bounds and some of their
consequences that we will use in this paper. The discrete transport
distance (or its associated Riemannian structure) has been introduced
independently in \cite{Ma11,Mie11a}. The notion of entropic Ricci
curvature bounds for Markov chains has been introduced and studied in
\cite{EM12}.

\subsection{Discrete transport distance and Ricci bounds}
\label{sec:def}

Let $\cX$ be a finite set and let $Q:\cX\times\cX\to\R_+$ be a
collection of transition rates. Then the operator $L$ acting on
functions $\psi : \cX \to \R$ via
\begin{align*}
L \psi(x) = \sum_{y \in \cX} Q(x,y) \big(\psi(y) - \psi(x)\big)  
\end{align*}
is the generator of a continuous time Markov chain on $\cX$. We make
the convention that $Q(x,x)=0$ for all $x\in\cX$. We shall assume that
$Q$ is irreducible, i.e.~ for all $x,y\in\cX $ there exist points
$(x_1=x,x_2,\dots,x_n=y)$ such that $Q(x_i,x_{i+1})>0$ for
$i=1,\dots,n-1$. This implies that there exists a unique stationary
probability measure $\pi$ on $\cX$, i.e.~ satisfying
\begin{align*}
\sum_{x\in\cX}Q(x,y)\pi(x) = \pi(y)\;.
\end{align*}
We shall further assume that $Q$ is reversible w.r.t.~ $\pi$ i.e.~ the
detailed-balance condition holds:
\begin{align}\label{eq:reversibility}
  Q(x,y)\pi(x) = Q(y,x)\pi(y)\quad \forall x.y\in\cX\;.
\end{align}
Since $\pi$ is strictly positive, we can identify the set of
probability measures on $\cX$ with the set of probability densities
w.r.t.~ $\pi$ denoted by
\begin{align*}
\PX=\{\rho\in\R_+^{\cX}:\sum_x\rho(x)\pi(x)=1\}\;.
\end{align*}

We consider a distance $\cW$ on $\PX$  defined for $\rho_0,\rho_1 \in \PX$ by
\begin{align*}
 \cW(\rho_0, \rho_1)^2
   := \inf_{\rho, \psi} 
   \bigg\{  \frac12   \int_0^1 
  \sum_{x,y\in \cX} (\psi_t(x) - \psi_t(y))^2
    		 \theta\big(\rho(x),\rho(y)\big)  Q(x,y)\pi(x)
      \dd t 
          \bigg\}\;,
\end{align*}
where the infimum runs over all sufficiently regular curves $\rho :
[0,1] \to \PX$ and $\psi : [0,1] \to \R^\cX$ satisfying the
continuity equation
\begin{align} \label{eq:cont}
 \begin{cases}
 \displaystyle\ddt \rho_t(x) 
   + \displaystyle\sum_{y \in \cX} ( \psi_t(y) - \psi_t(x) ) \theta\big(\rho(x),\rho(y)\big) Q(x,y) ~=~0\qquad \forall x \in \cX\;, \\ 
  \rho|_{t=0} = \rho_0\;, \qquad \rho|_{t=1}  = \rho_1\;.
 \end{cases}
\end{align}
Here $\theta$ denotes the logarithmic mean given by
\begin{align*}
  \theta(s,t) = \int_0^1s^\alpha t^{1-\alpha}\dd\alpha\;.
\end{align*}
It has been shown in \cite{Ma11} that $\cW$ defines a distance on
$\PX$. It turns out that it is induced by a Riemannian structure on
the interior $\PXs$ consisting of all strictly positive probability
densities. The distance $\cW$ can be seen as a discrete analogue of
the Benamou--Brenier formulation \cite{BB00} of the continuous
$L^2$-transportation cost. The appearance of the logarithmic mean is
due to the fact that it allows one to obtain a discrete chain rule for
the logarithm, namely $\hat\rho\nabla\log\rho=\nabla\rho$, where we
write $\nabla\psi(x,y)=\psi(y)-\psi(x)$ and
$\hat\rho(x,y)=\theta\big(\rho(x),\rho(y)\big)$. This replaces the
usual identity $\rho\nabla\log\rho=\nabla\rho$. The distance $\cW$ is
tailor-made in this way such that the discrete heat equation
$\partial_t\rho=L\rho$ is the gradient flow of the relative entropy
\begin{align*}
 \cH(\rho) = \sum_{x \in \cX} \pi(x) \rho(x) \log \rho(x)
\end{align*}
w.r.t.~ the Riemannian structure induced by $\cW$
\cite{Ma11,Mie11a}. This makes $\cW$ a natural replacement of the
Wasserstein distance in the discrete setting. Moreover, it has been
proven in \cite{EM12} that every pair of densities $\rho_0, \rho_1 \in
\PX$ can be joined by a constant speed $\cW$-geodesic
$(\rho_s)_{s\in[0,1]}$. Here constant speed geodesic means that
$\cW(\rho_s,\rho_t)=|s-t|\cW(\rho_0,\rho_1)$ for all $s,t\in[0,1]$.

In analogy with the approach of Lott--Sturm--Villani, the following
definition of Ricci curvature lower bounds has been given in \cite{EM12}.

\begin{definition}\label{def:intro-Ricci}
  $(\cX,Q,\pi)$ has \emph{Ricci curvature bounded from below by
    $\kappa \in \R$} if for any constant speed geodesic $\{\rho_t\}_{t
    \in [0,1]}$ in $(\PX, \cW)$ we have
  \begin{align*}
    \cH(\rho_t) \leq (1-t) \cH(\rho_0) + t \cH(\rho_1) -
    \frac{\kappa}{2} t(1-t) \cW(\rho_0, \rho_1)^2\;.
  \end{align*}
  In this case, we write $\Ric(\cX,Q,\pi) \geq \kappa$.
\end{definition}

\subsection{Equivalent formulation via Bochner-type inequality}
\label{sec:BA}

Entropic curvature bounds can be expressed more explicitly in terms of
an inequality resembling Bochner's inequality in Riemannian
geometry. To this end, let us briefly describe the Riemannian structure
induced by $\cW$.

At each $\rho\in\cP_*(\cX)$ the tangent space to $\cP_*(\cX)$ is given
by $\mathcal{T}=\{s\in\R^\cX:\sum_xs(x)\pi(x)=0\}$. Given
$\psi\in\R^\cX$ we denote by $\nabla\psi\in\R^{\cX\times\cX}$ the quantity $\psi(x,y)=\psi(y)-\psi(x)$, which is the discrete gradient of $\psi$. Fix
$x_0\in\cX$ and let
$\mathcal{G}=\{\nabla\psi:\psi\in\R^\cX,\psi(x_0)=0\}$ denote the set
of all discrete gradient fields modulo constants. It has been shown in
\cite[Sec.~3]{Ma11} that for each $\rho\in\cP_*(\cX)$, the map
\begin{align*}
  K_\rho: \nabla\psi \mapsto \sum_y\nabla\pi(y,x)Q(x,y)\,,
\end{align*}
defines a linear bijection between $\mathcal{G}$ and the tangent space
$\mathcal{T}$. One can then define a Riemannian metric tensor on
$\cP_*(\cX)$ by using this identification and introducing the scalar
product $\langle\cdot,\cdot\rangle_\rho$ on $\mathcal{G}$ depending on $\rho$ and
given by
\begin{align*}
  \langle\nabla\psi,\nabla\phi\rangle_\rho=\frac12\sum_{x,y}\nabla\psi(x,y)\nabla\phi(x,y)Q(x,y)\pi(x)\;.
\end{align*}
Then $\cW$ is the Riemannian distance associated to this Riemannian
structure. We will use the notation
$\cA(\rho,\psi):=\norm{\nabla\psi}_\rho^2$.

Entropic Ricci bounds, i.e.~convexity of the entropy along
$\cW$-geodesics, are determined by bounds on the Hessian of the entropy
$\cH$ in the Riemannian structure defined above. An explicit
expression of the Hessian at $\rho\in\cP_*(\cX)$ is given by
\begin{align*}
  \Hess\cH(\rho)[\nabla\psi] = \frac12\sum_{x,y}\left[\frac12\hat
    L\rho (x,y)|\nabla\psi(x,y)|^2-\hat\rho(x,y)\nabla\psi(x,y)\nabla
    L\psi(x,y)\right]Q(x,y)\pi(x)\;,
\end{align*}
where we have used the notation
\begin{align*}
\hat\rho(x,y) ~&:=~ \theta(\rho(x),\rho(y))\;,\\
 \hat L \rho(x,y) ~&:=~  
 \partial_1\theta\big(\rho(x),\rho(y)\big) L\rho(x) 
  + \partial_2\theta\big(\rho(x),\rho(y)\big) L\rho(y)\;.
\end{align*}
Setting 
$\cB(\rho,\psi):=\Hess\cH(\rho)[\nabla\psi]$ for brevity, we then have the
following equivalent characterization of entropic Ricci bounds.

\begin{proposition}[{\cite[Thm.~4.4]{EM12}}]\label{prop:Ric-equiv}
  A Markov triple $(\cX,Q,\pi)$ satisfies $\Ric(\cX,Q,\pi)\geq\kappa$
  if and only if for every $\rho\in\cP_*(\cX)$ and every
  $\psi\in\R^\cX$ we have
  \begin{align*}
    \cB(\rho,\psi)~\geq~\kappa \cA(\rho,\psi)\;.
  \end{align*}
\end{proposition}

Note that this statement is non-trivial since the Riemannian metric
degenerates at the boundary of $\PX$. In view of the explicit
expressions of $\cA$ and $\cB$, the criterion above closely resembles
(an integrated version of) the classical Bochner inequality or
Bakry--\'Emery $\Gamma_2$-criterion. Namely, a Riemannian manifold $M$
satisfies $\Ric\geq\kappa$ if and only if for every pair of smooth
functions $\rho,\psi:M\to\R$ we have:
\begin{align*}
\int_M \frac12 \left[L\rho |\nabla\psi|^2 -\rho\langle\nabla\psi,\nabla L\psi\rangle \right] \dd\mathrm{vol}
~\geq~ \int_M\rho|\nabla\psi|^2\dd\mathrm{vol}\;,  
\end{align*}
where $\nabla$ now denotes the usual gradient and $L$ denotes the
Laplace--Beltrami operator. In fact, the left hand side equals the
Hessian of the entropy in Otto's formal Riemannian structure on
$\cP(M)$ associated with the $L^2$-Wasserstein distance $W_2$.

\subsection{Functional inequalities and trend to equilibrium}

Entropic Ricci curvature lower bounds have many consequences in terms
of functional inequalities as was shown in \cite[Sec.~7]{EM12}. More
precisely, if a Markov triple $(\cX,Q,\pi)$ satisfies
$\Ric(\cX,Q,\pi)\geq\kappa$ with $\kappa >0$ then the following hold:
\begin{itemize}
\item a modified logarithmic Sobolev inequality MLSI($\kappa$):
  \begin{align}\label{eq:MLSI}
    \cH(\rho) \leq \frac{1}{2\kappa} \cE(\rho,\log\rho)\quad \forall\rho\in\PXs\;,
  \end{align}
\item a modified Talagrand inequality T$_\cW(\kappa)$:
  \begin{align}\label{eq:TW}
    \cW(\rho,\one)^2 \leq \frac{2}{\kappa} \cH(\rho)\quad \forall\rho\in\PX\;,
  \end{align}
\item a Poincar\'e inequality P($\kappa$):
  \begin{align}\label{eq:Poinc}
     \mathrm{Var}_\pi(\psi) \leq \frac1\kappa \cE(\psi,\psi)\quad \forall \psi\;,
  \end{align}
\end{itemize}
where $\mathrm{Var}_\pi(\psi)=\pi[\psi^2]-\pi[\psi]^2$ and $\cE$ is a
discrete Dirichlet form given as
\begin{align*}
  \cE(\psi,\phi) = \frac12\sum_{x,y\in\cX} \big(\psi(y)-\psi(x)\big)\big(\phi(y)-\phi(x)\big)Q(x,y)\pi(x)\;.
\end{align*}

It is well known that the modified logarithmic Sobolev inequality and
the Poincar\'e inequality govern the trend to equilibrium of the
Markov semigroup $P_t=e^{tL}$. Indeed, noting that
\begin{align*}
  \ddt\cH(P_t\rho)=-\cE(P_t\rho,\log P_t\rho)\;, \qquad \ddt\mathrm{Var}(P_t\psi)=-\cE(P_t\psi,P_t\psi)\;,
\end{align*}
the Gronwall lemma together with the inequalities \eqref{eq:MLSI} and
\eqref{eq:Poinc} yield the exponential convergence estimates
\begin{align*}
   \cH(P_t\rho)\leq e^{-2\kappa t}\cH(\rho)\;.
   \qquad 
   \mathrm{Var}(P_t\psi)\leq e^{-\kappa t}\mathrm{Var}(\psi)\;,
\end{align*}

Let us make the connection to the notion of \emph{convex entropy
  decay} and the Bakry--\'Emery approach to the MLSI developed in the
discrete setting in \cite{CDPP,DPP}. This approach is based on the
following observation (see \cite{CDPP}):
\begin{lemma}
 Let $\kappa>0$ and assume that the convex entropy decay inequality
 \begin{align}\label{eq:ced}
   \sum_x\left[L\rho(x)L\log\rho(x) + \frac{(L\rho)^2}{\rho}\right]\pi(x) \geq \kappa \cE(\rho,\log\rho)
 \end{align}
 holds for all $\rho\in\cP_*(\cX)$. Then MLSI$(\kappa)$ holds.
\end{lemma}
The idea is that to note that
\[
	\frac{\mathrm{d}^2}{\mathrm{d}t^2}\cH(P_t\rho)=\sum_x\left[LP_t\rho(x)L\log
  P_t\rho(x) + \frac{(LP_t\rho)^2}{P_t\rho}\right]\pi(x).
\]
Thus, \eqref{eq:ced} asserts that
\[\frac{\mathrm{d}^2}{\mathrm{d}t^2}\cH(P_t\rho)\leq
-\kappa\frac{\mathrm{d}}{\mathrm{d}t}\cH(P_t\rho).
\]
After integration, this inequality yields $\frac{\mathrm{d}}{\mathrm{d}t}\cH(P_t\rho)\leq
-\kappa \cH(P_t\rho)$, and thus~ MLSI$(\kappa)$.

Now, a direct calculation reveals that 
\begin{align*}
  \cA(\rho,\log\rho) &= \cE(\rho,\log(\rho))\;,\\
   \cB(\rho,\log\rho) &= \sum_x\left[L\rho(x)L\log\rho(x) + \frac{(L\rho)^2}{\rho}\right]\pi(x)\;.
\end{align*}
Thus, we obtain that $\Ric(\cX,Q,\pi)\geq \kappa$ implies, in
particular, the convex entropy decay inequality \eqref{eq:ced}.

Finally, we recall that entropic Ricci bounds also imply exponential
contraction in the discrete transport distance $\cW$
\cite[Prop.~4.7]{EM12}. More precisely, if
$\Ric(\cX,Q,\pi)\geq\kappa$, then for all $\rho_0,\rho_1\in\PX$ we
have
\begin{align*}
  \cW(P_t\rho_0,P_t\rho_1)\leq e^{-\kappa t}\cW(\rho_0,\rho_1)\;.
\end{align*}

\section{A perturbative approach to Ricci bounds}
\label{sec:perturb}

In this section we present a general method to obtain entropic Ricci
bounds for systems of weakly interacting Markov chains. The method
starts from Proposition \ref{prop:Ric-equiv} and proceeds in two steps to
establish the inequality $ cB\geq\kappa\cA$. The first one consist in
reorganizing the $\cB$-term, identifying non-negative contributions
and giving a first lower bound by neglecting these. A general method
for this, the so called Bochner-Bakry-\'Emery approach, was developed
in \cite{DPP} in the study of spectral gap, MLSI and convex entropy
decay and was generalized in \cite{FM15} to the level of Ricci
curvature. We will recall this approach in Section \ref{sec:BBE} and
give a short simplified proof. The second step, detailed in Section
\ref{sec:perturb-sub}, constitutes our main result and gives a final bound on
$\cB$ using the fact that the interactions are weak.

Before we proceed, we introduce a different representation of the
Markov chain that will be convenient in the sequel. Let $G$ be a set
of maps from $\cX$ to itself (called allowed moves) and consider a
function $c:\cX\times G\to\R_+$ (called jump rates).

\begin{definition}\label{def:mappingrepresentation}
  We call the pair $(G,c)$ a \emph{mapping representation} of $Q$ if
  the following properties hold:
  \begin{enumerate}
  \item The generator $L$ can be written in the form
    \begin{align}\label{eq:generator-mappingform}
      L\psi(x)~=~\sum\limits_{\delta\in G} \nabla_\delta \psi(x)
      c(x,\delta)\ ,
    \end{align}
    where
    \begin{align*}
      \nabla_\delta \psi(x) = \psi(\delta x)-\psi(x)\;.
    \end{align*}
  \item For every $\delta\in G$ there exists a unique $\delta^{-1}\in
    G$ satisfying $\delta^{-1}(\delta(x))=x$ for all $x$ with
    $c(x,\delta)>0$.
  \item For every $F:\cX\times G\to\R$ we have
    \begin{align}\label{eq:rev}
      \sum\limits_{x\in\cX,\delta\in
        G}F(x,\delta)c(x,\delta)\pi(x)~=~\sum\limits_{x\in\cX,\delta\in
        G}F(\delta x,\delta^{-1})c(x,\delta)\pi(x)\ .
    \end{align}
  \end{enumerate}
\end{definition}

Note that the detailed-balance condition~\eqref{eq:reversibility} turns into
$$c(x,\delta)\pi(x)=c(\delta x,\delta^{-1})\pi(\delta x)\quad \forall x\in\cX,\ \delta\in G\;.$$
Every irreducible, reversible Markov chain has a mapping
representation. In fact, an explicit mapping representation can be
obtained as follows. For $x,y\in\cX$ consider the bijection
$t_{\{x,y\}}:\cX\to\cX$ that interchanges $x$ and $y$ and keeps all
other points fixed. Then let $G$ be the set of all these
``transpositions'' and set $c(x,t_{\{x,y\}})=Q(x,y)$ and
$c(x,t_{\{y,z\}})=0$ for $x\notin \{y,z\}$. Then $(G,c)$ defines a
mapping representation. However, in examples it is often more natural
to work with a different mapping representation involving a smaller
set $G$.

Using a mapping representation $(G,c)$ of $Q$, we can write out the
quantities $\cA$ and $\cB$ explicitly. We obtain
\begin{align}\label{eq:A-mappingform}
  \cA(\rho,\psi)~=~\frac12 \sum\limits_{x\in\cX,\delta\in G}\big(\nabla_\delta\psi(x)\big)^2\hat\rho(x,\delta x)c(x,\delta)\pi(x)\;.
\end{align}
Moreover, setting for convenience $\hrho_i(x,y)
:= \partial_i\theta(\rho(x),\rho(y))$ for $i=1,2$, we get
\begin{equation}\begin{aligned}
\label{eq:T1-mappingform}
  &\cB(\rho,\psi)~=~\\
  &\frac{1}{4}\sum\limits_{x\in\cX}\sum\limits_{\delta,\eta\in G}\big(\nabla_\delta\psi(x)\big)^2\bigg[\hat{\rho}_1(x,\delta
  x)\nabla_\eta\rho(x)c(x,\eta) +\hat{\rho}_2(x,\delta x)\nabla_\eta\rho(\delta x)c(\delta
  x,\eta)\bigg]c(x,\delta)\pi(x)\\
&\qquad\qquad-2\nabla_\delta\psi(x)\bigg[\nabla_\eta\psi(\delta x)c(\delta x,\eta)-\nabla_\eta\psi(x)c(x,\eta)\bigg]\hat{\rho}(x,\delta x) c(x,\delta)\pi(x)\\
&=~\frac12\sum\limits_{x,\delta,\eta}\Big[\abs{\nabla_\delta\psi}^2(x)\hat{\rho}_1(x,\delta
  x)\nabla_\eta\rho(x) - 2\nabla_\delta\psi(x)\nabla_\eta\psi(x)\hat{\rho}(x,\delta x)\Big]c(x,\eta)c(x,\delta)\pi(x)\;.
\end{aligned}\end{equation}
Here we have used reversibility and the fact that
$\hrho_1(x,y)=\hrho_2(y,x)$ in the last equality.

\begin{remark}\label{rmk:enlarge}
  It will be convenient sometimes to allow more flexibility in the
  mapping representation by considering a larger space
  $\cX'\supset\cX$ and a collection $G'$ of maps from $\cX'$ to
  itself. We trivially extend $\pi$ by $0$ to a probability measure on
  $\cX'$ and similarly the rates $Q$ to $\cX'\times\cX'$. $G'$
  together with a function $c':\cX'\times G'\to \R_+$ will still be
  called a mapping representation if all the properties of Definition
  \ref{def:mappingrepresentation} hold. In particular, we have
  $c(x,\delta)=0$ if $x$ or $\delta x$ belongs to $\cX'\setminus\cX$.
  Obviously, for any $\rho\in\cP(\cX'),\psi\in\R^{\cX'}$, the
  expressions in the right hand side of \eqref{eq:A-mappingform}and
  \eqref{eq:T1-mappingform} calculated with the extended mapping
  representation $(G',c')$ coincide with the original quantities
  $\cA(\rho|_{\cX},\psi|_{\cX})$ and $\cB(\rho|_{\cX},\psi|_{\cX})$.
\end{remark}

\subsection{The Bochner--Bakry--\'Emery approach to Ricci bounds}\label{sec:BBE}

Here we briefly recall the main result of \cite{FM15}, a general
method to identify non-negative contributions to the $\cB$-term.

For convenience, we give a short and simplified proof.

\begin{definition}\label{def:R}
  We call a function $R:\cX\times G\times G\to\R_+$ admissible for $Q$
  if (and only if)
   \begin{itemize}
    \item[(i)] $\delta\eta x=\eta\delta x$ for all $x,\delta,\eta$ with $R(x,\delta,\eta)>0$,
    \item[(ii)] $R(x,\delta,\eta)=R(x,\eta,\delta)$ for all $x,\delta,\eta$ with $c(x,\delta)c(x,\eta)>0$,~~and
    \item[(iii)] $R(x,\delta,\eta)=R(\delta x,\delta^{-1},\eta)$ for all $x,\delta,\eta$ with $c(x,\delta)c(x,\eta)>0$.
  \end{itemize}
\end{definition}

\begin{proposition}[{\cite[Thm.~3.5]{FM15}}]\label{prop:BBE}
  Let $R$ be admissible for $Q$ and define $\Gamma:\cX\times G\times
  G\to\R$ via
  $\Gamma(x,\delta,\eta)=c(x,\delta)c(x,\eta)\pi(x)-R(x,\delta,\eta)$. Then
  we have
  \begin{align}\label{eq:Gamma}
    \cB(\rho,\psi)
    ~&\geq~
  \sum\limits_{x,\delta,\eta}\Gamma(x,\delta,\eta)\left[\frac12\abs{\nabla_\delta\psi}^2(x)\hat{\rho}_1(x,\delta
  x)\nabla_\eta\rho(x) + \nabla_\delta\psi(x)\nabla_\eta\psi(x)\hat{\rho}(x,\delta x)\right]\;.
  \end{align}
\end{proposition}

\begin{proof}
  The proof works verbatim as \cite[Prop. 5.4]{EM12}, using the
  properties (i)-(iii) of Definition~\ref{def:R}, instead of the
  conditions on $c$ given there.
  Recalling \eqref{eq:T1-mappingform} it suffices to show that
  \begin{align}\label{eq:Rtrick}
  B:=\sum\limits_{x,\delta,\eta}R(x,\delta,\eta)\left[\frac12\abs{\nabla_\delta\psi}^2(x)\hat{\rho}_1(x,\delta
  x)\nabla_\eta\rho(x) + \nabla_\delta\psi(x)\nabla_\eta\psi(x)\hat{\rho}(x,\delta x)\right]~\geq~0\;.
 \end{align}
 We first use (iii) to symmetrize in $x$ and $\delta x$ and obtain
 \begin{align*}
   B = \frac12\sum\limits_{x,\delta,\eta}R(x,\delta,\eta)&\Big[\frac12\abs{\nabla_\delta\psi}^2(x)\big[\hat{\rho}_1(x,\delta
  x)\nabla_\eta\rho(x) +\hat{\rho}_2(x,\delta x)\nabla_\eta\rho(\delta x)\big]\\
    &+ \nabla_\delta\psi(x)\big[\nabla_\eta\psi(x)-\nabla_\eta\psi(\delta x)\big]\hat{\rho}(x,\delta x)\Big]\;.
 \end{align*}
 In the first term we use the (in-)equalities \eqref{eq:lm-trick1} and
 \eqref{eq:lm-trick2}, while in the second term we use (i) and the
 fact that $\nabla_\eta\psi(x)-\nabla_\eta\psi(\delta x)=
 \nabla_\delta\psi(x)-\nabla_\delta\psi(\eta x)$ provided $\delta\eta
 x=\eta\delta x$. This yields
 \begin{align*}
   B= \frac14 \sum\limits_{x,\delta,\eta}R(x,\delta,\eta)\Big[\abs{\nabla_\delta\psi}^2(x)\big[\hat{\rho}(\eta x,\delta
  \eta x) + \hat{\rho}(x,\delta x)\big]
    + \nabla_\delta\psi(x)\nabla_\delta\psi(\eta x)\hat{\rho}(x,\delta x)\Big]\;.
 \end{align*}
 Finally, we use (iii) again to symmetrize in $x$ and $\eta x$, and
 complete the square to get
 \begin{align*}
    B&= \frac18 \sum\limits_{x,\delta,\eta}R(x,\delta,\eta)\Big[\abs{\nabla_\delta\psi}^2(x) + \abs{\nabla_\delta\psi}^2(\eta x)
           + \nabla_\delta\psi(x)\nabla_\delta\psi(\eta x)\Big]\big[\hat{\rho}(\eta x,\delta\eta x) + \hat{\rho}(x,\delta x)\big]\\
     &\geq
        \frac1{16} \sum\limits_{x,\delta,\eta}R(x,\delta,\eta)\big|\nabla_\delta\psi(x) + \nabla_\delta\psi(\eta x)\big|^2
        \big[\hat{\rho}(\eta x,\delta\eta x) + \hat{\rho}(x,\delta x)\big]
      \geq 0\;,
 \end{align*}
 which finishes the proof.
\end{proof}

\subsection{The perturbative criterion}\label{sec:perturb-sub}
Here we present our main result: a general entropic Ricci bound for
weakly interacting Markov chains (see Theorems \ref{thm:perturb} and
\ref{thm:Cayley-Ricci} below).

We start by introducing the following notation. For any $\psi\in\R^\cX$
and $\rho\in\PX$ we write
\begin{align}\label{eq:Bterms}
  B(\rho,\psi)(x,\delta,\eta) ~:=~ \frac12\abs{\nabla_\delta
    \psi}^2(x)\hat\rho^1(x,\delta x)\nabla_\eta\rho(x) +
  \nabla_\delta\psi(x) \nabla_\eta\psi(x)\hat\rho(x,\delta x)\;.
\end{align}
We will often suppress the dependence on $\rho,\psi$, writing simply
$B(x,\delta,\eta)$, if no confusion can arise. Note that with this
notation
\begin{align*}
  \cB(\rho,\psi)~=~\sum\limits_{x\in\cX}\sum\limits_{\delta,\eta\in G}B(\rho,\psi)(x,\delta,\eta)c(x,\delta)c(x,\eta)\pi(x)\;.
\end{align*}

In this sum, we distinguish between two types of contributions,
namely \emph{diagonal} contributions of the form
$B(\rho,\psi)(x,\delta,\delta)$ and \emph{off-diagonal} contributions
of the form $B(\rho,\psi)(x,\delta,\eta)$ with $\eta\neq\delta$.  In
the proof of our main result we obtain a lower bound on $\cB$ using
three ingredients. We will first show in Lemma \ref{lem:B-diagonal}
that the diagonal part of $\cB$ always gives a positive contribution
to curvature. Secondly, provided the interactions are sufficiently
weak, expressed through a quantitative assumption on deviation of the
jump rates from being homogeneous, we can use the method from the
previous section and techniques developed in \cite{EMT15} to discard a
large fraction of the off-diagonal contributions. Finally, Lemma
\ref{lem:pointwise} will allow us to estimate the remaining
off-diagonal contributions against the corresponding diagonal
contributions.

In the sequel we will use the following properties of the logarithmic
mean, see e.g.~\cite[Lem.~2.2]{EM12}:
\begin{lemma}\label{lem:lm-trick}
  For any $s,t,u,v>0$ we have:
  \begin{align}\label{eq:lm-trick1}
    u\partial_1\theta(u,v) +
    v\partial_2\theta(u,v)~&=~\theta(u,v)\;,\\\label{eq:lm-trick2}
    u\partial_1\theta(s,t) +
    v\partial_2\theta(s,t)~&\geq~\theta(u,v)\;.
  \end{align}
\end{lemma}

We have the following bounds on the \emph{on-diagonal} part of $\cB$.

\begin{lemma}\label{lem:B-diagonal}
  For all $\rho\in\R_+^\cX$ and $\psi\in\R^\cX$ we have that
  $B(\rho,\psi)(x,\delta,\delta)\geq0$ for all $x\in\cX$ and
  $\delta\in G$ and it holds:
  \begin{align}\label{eq:B-diagonal}
    \sum\limits_{x\in\cX,\delta\in G} B(\rho,\psi)(x,\delta,\delta)c(x,\delta)\pi(x) ~\geq~ 2\cA(\rho,\psi)\;.
  \end{align}
  Let $H$ be a subset of $G$ such that $G = H^{-1} \cup H$.  Then, we have that
  \begin{equation}\label{eq:B-half-diagonal}
  \sum\limits_{x\in \cX,\delta\in H} B(\rho,\psi)(x,\delta,\delta)c(x,\delta)\pi(x) ~\geq~ \frac{1}{2} \cA(\rho,\psi)\;.	
  \end{equation}
\end{lemma}

\begin{proof}
 First, we calculate that
  \begin{align*}
    \sum\limits_{x\in\cX,\delta\in G} B(\rho,\psi)&(x,\delta,\delta)c(x,\delta)c(x,\delta)\pi(x)\\
    =~&\sum\limits_{x,\delta} \frac12\abs{\nabla_\delta\psi}^2(x)c(x,\delta)c(x,\delta)\pi(x)
   \left[\hat\rho_1(x,\delta x)\rho(\delta x) + \hat\rho_2(x,\delta x)\rho(\delta x) + \hat\rho(x,\delta x) \right]\\
     =~&\cA(\rho,\psi) + \sum\limits_{x,\delta} \frac12\abs{\nabla_\delta\psi}^2(x)c(x,\delta)\pi(x)
   \left[\hat\rho_1(x,\delta x)\rho(\delta x) + \hat\rho_2(x,\delta x)\rho(\delta x)\right]\,.
  \end{align*}
  For the second term in the last line we use reversibility, the fact
  that $\partial_1\theta(s,t)=\partial_2\theta(t,s)$ and
  \eqref{eq:lm-trick2} and obtain
  \begin{align*}
    &\sum\limits_{x,\delta} \frac12\abs{\nabla_\delta\psi}^2(x)c(x,\delta)\pi(x)
   \left[\hat\rho_1(x,\delta x)\rho(\delta x) + \hat\rho_2(x,\delta x)\rho(\delta x)\right]\\
   =~&\sum\limits_{x,\delta} \frac14\abs{\nabla_\delta\psi}^2(x)c(x,\delta)\pi(x)
   \Big[\hat\rho_1(x,\delta x)\big(\rho(\delta x)+\rho(x)\big) + \hat\rho_2(x,\delta x)\big(\rho(\delta x)+\rho(x)\big)\Big]\\
  \geq~&\sum\limits_{x,\delta} \frac12\abs{\nabla_\delta\psi}^2(x)c(x,\delta)\pi(x)\hat\rho(x,\delta(x)) ~\geq~ \cA(\rho,\psi)\;.
  \end{align*}
  
  To obtain \eqref{eq:B-half-diagonal}, we first use
  \eqref{eq:lm-trick1} and \eqref{eq:lm-trick2} to see that
 \begin{equation}
 \begin{split}\label{eq:B-nonnegative}
   B(\rho,\psi)(x,\delta,\delta)~&=~\frac12\abs{\nabla_\delta\psi}^2(x)
   \left[\hat\rho_1(x,\delta x)\big(\rho(\delta x)-\rho(x)\big)+2\hat\rho(x,\delta x) \right]\\
    &=~\frac12\abs{\nabla_\delta\psi}^2(x)
\left[\hat\rho_1(x,\delta x)\rho(\delta x) + \hat\rho_2(x,\delta x)\rho(\delta x) + \hat\rho(x,\delta x) \right]\\
	&\geq ~\frac12 |\nabla_\delta \psi|^2(x) \hat\rho(x,\delta x)\;,
 \end{split}
 \end{equation}
 which is non-negative. Then, notice that by symmetrization and
 reversibility, we have that
 \begin{align*}
   \mathcal{A}(\rho,\psi) \leq \sum_{\delta \in H} \sum_{x\in\cX}
   |\nabla_\delta\psi (x)|^2 \hat\rho(x, \delta x) c(x, \delta)
   \pi(x)\;,
 \end{align*}
 which together with \eqref{eq:B-nonnegative} immediately yields
 \eqref{eq:B-half-diagonal}.
\end{proof}

We will use the following to estimate the off-diagonal contributions to
$\cB$. Similar estimates for terms appearing in the study of convex
entropy decay can be found in \cite[(2.33)]{DPP}.

\begin{lemma}\label{lem:pointwise}
  For any $\psi\in\R^\cX$ and $\rho\in\PX$ and $x\in\cX, \delta,\eta\in G$ we have
  \begin{align}\label{eq:pointwise}
    B(x,\delta,\eta) + B(x,\eta,\delta)
    ~\geq~
    - B(\delta x ,\delta^{-1},\delta^{-1}) - B(\eta x,\eta^{-1},\eta^{-1})\;. 
  \end{align}
\end{lemma}

\begin{proof}
  Setting $a=\nabla_\delta\psi(x)$, $b=\nabla_\eta\psi(x) $ as well as
  $s=\rho(x)$, $t=\rho(\delta x)$ and $r=\rho(\eta x)$, it suffices to
  show that
  \begin{align*}
    &a^2\Big[\partial_1\theta(s,t)(r-s)+\partial_1\theta(t,s)(s-t)+2\theta(t,s)\Big]
   +2 ab \Big[\theta(s,t)+\theta(s,r)\Big] \\
   &+ b^2\Big[\partial_1\theta(s,r)(t-s)+\partial_1\theta(r,s)(s-r)+2\theta(r,s)\Big]
~\geq~
   0\;.
  \end{align*}
  We rewrite this last inequality as $a^2M_{11} + 2ab M_{12} + b^2
  M_{22}\geq 0$ with a symmetric $2 \times 2$ matrix $M$. Now, it is
  readily checked, using the fact that
  $\partial_1\theta(u,v)=\partial_2\theta(v,u)$ as well as
  \eqref{eq:lm-trick1}, \eqref{eq:lm-trick2}, that $M$ is diagonally
  dominant and thus non-negative definite.
\end{proof}

\subsubsection{Commutative mapping
  representations}\label{sec:commutative}

Let $(\cX,Q,\pi)$ be a Markov triple and assume that it has a mapping
representation $(G,c)$ that is \emph{commutative} in the sense that
\begin{align*}
  \delta\circ\eta = \eta\circ\delta \qquad\forall \delta,\eta\in G\;.
\end{align*}

A first criterion for entropic Ricci bounds in this setting was given
in \cite{EM12}.
\begin{proposition}[{\cite[Prop.~5.4]{EM12}}]\label{prop:EM-no-interact}
  Assume that 
  \begin{align}\label{eq:rates-homo}
    c(\delta x,\eta) = c(x,\eta)\qquad \forall x,\in\cX,
    \delta,\eta\in G\;.
  \end{align}
  Then, we have $\Ric(\cX,Q,\pi)\geq 0$. If moreover
  $\delta=\delta^{-1}$ holds for all $\delta\in G$, the we have
  $\Ric(\cX,Q,\pi)\geq 2c_*$, where
  \begin{align}\label{eq:min-rate}
    c_*~:=~ \min\{c(x,\delta)~:~x,\delta~\text{with}~c(x,\delta)>0 \}
  \end{align}
  denotes the minimal transition rate.
\end{proposition}
Condition \eqref{eq:rates-homo} is a requirement on the transition
rates to be homogeneous. Our main result, Theorem~\ref{thm:perturb}, of this section is a perturbative generalization of this criterion,
when an explicit bound on the non-homogeneity of the transition rates
is given.

To state the result, we use the following notation. Put
$q(x,\delta,\eta)=c(x,\delta)c(x,\eta)\pi(x)$. For $\delta,\eta\in G$
with $\eta\neq\delta,\delta^{-1}$ we
define
\begin{align}\label{eq:qstar1}
   q_*(x,\delta,\eta)~&:=~\min\Big\{q(x,\delta,\eta),q(\delta x,\delta^{-1},\eta),q(\eta x,\delta,\eta^{-1}),q(\delta \eta x,\delta^{-1},\eta^{-1}) \Big\}\;.
\end{align}

\begin{theorem}\label{thm:perturb}
 Assume that 
 \begin{align}\label{eq:ass-rates}
  \lambda~:=~\min_{\stackrel{x\in\cX,\delta\in G}{c(x,\delta)>0}}\left[c(x,\delta)-\one_{\delta\neq\delta^{-1}}c(\delta x,\delta) - \sum_{\eta:\eta\neq\delta,\delta^{-1}}
   \frac{(q-q_*)(\delta x,\delta^{-1},\eta)}{c(x,\delta)\pi(x)}\right]~\geq~0\;,
 \end{align}
 Then, we have $\Ric(\cX,Q,\pi)\geq 2\lambda$.
 
 Moreover, assume that there are disjoint subsets $H_1,H_2$ of $G$ such that $H_1\cap H_2=\emptyset$
 and $H_i\cup H_i^{-1}=G$ for $i=1,2$. Set
 \begin{align}\label{eq:ass-rates2}
   \lambda_i~:=~\min_{\stackrel{x\in\cX,\delta\in H_i}{c(x,\delta)>0}}\left[c(x,\delta)-\one_{\delta\neq\delta^{-1}}c(\delta x,\delta) - \sum_{\eta:\eta\neq\delta,\delta^{-1}}
   \frac{(q-q_*)(\delta x,\delta^{-1},\eta)}{c(x,\delta)\pi(x)}\right]\;.
 \end{align}
 Then, we also have $\Ric(\cX,Q,\pi)\geq \frac{1}{2}(\lambda_1+\lambda_2)$.
\end{theorem}

Note that we recover Proposition \ref{prop:EM-no-interact} as an
immediate consequence: In this situation we have $q-q_*\equiv0$ and
hence $\lambda=0$ or $2c_*$, depending on whether there is $\delta$ with
$\delta\neq\delta^{-1}$ or not.

\begin{proof}
  To prove the first statement, we have to show that for any $\rho$
  and $\psi$\,,
\begin{align}\label{eq:Ric_Ising-BA}
  \cB(\rho,\psi)~\geq~ 2\lambda\cA(\rho,\psi)\;.
\end{align}
Define a function $R:\cX\times G\times G\to\R_+$ as follows.
For $\delta,\eta\in G$ with $\eta\neq\delta,\delta^{-1}$ set
\begin{align*}
  R(x,\delta,\eta)~=~ q_*(x,\delta,\eta)\;,
\end{align*}
and for $\delta\in G$ with $\delta\neq \delta^{-1}$ set
\begin{align*}
 R(x,\delta,\delta^{-1})~&=~ q(x,\delta,\delta^{-1})\;,\\ 
 R(x,\delta,\delta)~&=~ c(\delta x, \delta) c(\delta x,\delta^{-1})\pi(\delta x) = q(x,\delta,\delta)\frac{c(\delta x,\delta)}{c(x,\delta)}\;.
 \end{align*}
 It is readily checked that $R$ is admissible in the sense of
 Definition \ref{def:R}. Note that the assumption on $\lambda$ guarantees, in
 particular, that $c(\delta x,\delta)\leq c(x,\delta)$ when $\delta \neq \delta^{-1}$. Thus, we have
 that $\Gamma(x,\delta,\eta)=q(x,\delta,\eta)-R(x,\delta,\eta)\geq 0$
 for all $x,\delta,\eta$. Note further that in the case $\delta\neq
 \delta^{-1}$, we have $\Gamma(x,\delta,\delta^{-1})=0$. Let us write
 for brevity $B(x,\delta,\eta):=B(\rho,\psi)(x,\delta,\eta)$. Using
 Proposition \ref{prop:BBE} and Lemma \ref{lem:pointwise} we now
 obtain
\begin{align*}
  \cB(\rho,\psi)
   &~\geq~
  \sum\limits_{x\in\cX,\delta,\eta\in G}\Gamma(x,\delta,\eta)B(x,\delta,\eta)\\
  &=~
   \sum\limits_{x,\delta}B(x,\delta,\delta)\Gamma(x,\delta,\delta) +  \frac12\sum\limits_{x,\delta\neq \eta}\big[B(x,\delta,\eta)+B(x,\eta,\delta)\big]\Gamma(x,\delta,\eta)\\
  &\geq~
   \sum\limits_{x,\delta}B(x,\delta,\delta)\Gamma(x,\delta,\delta) -  \frac12\sum\limits_{x,\delta\neq \eta}\big[B(\delta x,\delta^{-1},\delta^{-1})+B(\eta x,\eta^{-1},\eta^{-1})\big]\Gamma(x,\delta,\eta)\\
  &=~
   \sum\limits_{x,\delta}B(x,\delta,\delta)\Gamma(x,\delta,\delta) - \sum\limits_{x,\delta\neq \eta}B(\delta x,\delta^{-1},\delta^{-1})\Gamma(x,\delta,\eta)\;.
\end{align*}
Here we have also used in the second inequality the fact that
$B(x,\delta,\delta)\geq0$, by Lemma~\ref{lem:B-diagonal}. We can
further reorganize this expression to obtain
\begin{align*}
 \cB(\rho,\psi) &\geq~
   \sum\limits_{x,\delta}B(x,\delta,\delta)\left[\Gamma(x,\delta,\delta) - \sum\limits_{\eta:\eta\neq\delta^{-1}}\Gamma(\delta x,\delta^{-1},\eta)\right]\\
  &=~ \sum\limits_{x,\delta}B(x,\delta,\delta)\left[q(x,\delta,\delta) - \one_{\{\delta\neq\delta^{-1}\}}q(\delta x,\delta^{-1},\delta)-\sum\limits_{\eta:\eta\neq\delta,\delta^{-1}}(q-q_*)(\delta x,\delta^{-1},\eta)\right]\\
&\geq~ \sum\limits_{x,\delta}B(x,\delta,\delta)c(x,\delta)\pi(x)\left[c(x,\delta) - \one_{\{\delta\neq\delta^{-1}\}}c(\delta x,\delta)-\sum\limits_{\eta:\eta\neq\delta,\delta^{-1}}\frac{(q-q_*)(\delta x,\delta^{-1},\eta)}{c(x,\delta)\pi(x)}\right]\;.
\end{align*}
Now, invoking \eqref{eq:ass-rates} and \eqref{eq:B-diagonal} from
Lemma \ref{lem:B-diagonal} finishes the proof of statement i).

To obtain the second statement, we proceed in the same way. In the
last step, we note that by \eqref{eq:ass-rates} each summand is
non-negative. Thus we obtain the estimate
\begin{align*}
  \cB(\rho,\psi) 
  \geq
 \sum\limits_{x,\delta\in H_1\cup H_2}B(x,\delta,\delta)c(x,\delta)\pi(x)\left[c(x,\delta) - \one_{\{\delta\neq\delta^{-1}\}}c(\delta x,\delta)-\sum\limits_{\eta:\eta\neq\delta,\delta^{-1}}\frac{(q-q_*)(\delta x,\delta^{-1},\eta)}{c(x,\delta)\pi(x)}\right]\;,
\end{align*}
and we conclude by invoking \eqref{eq:ass-rates2} and
\eqref{eq:B-half-diagonal}.
\end{proof}

In Section~\ref{sec:examples}, we will apply the first part of
Theorem~\ref{thm:perturb} to derive lower Ricci bounds for the Glauber
dynamics of the Ising model. The second part of
Theorem~\ref{thm:perturb} is applied to derive lower Ricci bounds for
the hard-core model.

The following corollary will illustrate that our method allows to
obtain rough entropic Ricci curvature bounds under very explicit and
easy-to-check conditions on the transition rates. In practice,
however, a direct application of Theorem \ref{thm:perturb} will give
sharper results.

Assume for simplicity that $\delta=\delta^{-1}$ for all $\delta\in G$
and set, using the convention that $0/0=0$,
\begin{align*}
  N &:= \#\Big\{ \{\delta,\eta\}\subset G :  c(\delta x,\eta) \neq c(x,\eta) \text{ for some } x\in\cX   \Big\}\;,\\
  \alpha &:= \max\Big\{\log \frac{c(\delta x,\eta)}{c(x,\eta)} : x\in\cX,\delta,\eta\in G \text{ with } c(x,\eta)>0\Big\}\;,\\
  \beta &:= \max\Big\{\frac{c(x,\eta)}{c(x,\delta)} : x\in\cX,\delta,\eta\in G \text{ with } c(x,\delta)>0\Big\}\;.
\end{align*}

\begin{corollary}\label{cor:perturb}
  With the above notation, assume that
  \begin{align*}
  \varepsilon:=  \beta N \Big(e^{2\alpha}-1\Big)\leq 1\;.
  \end{align*}
  Then, we have $\Ric(\cX,Q,\pi)\geq (1-\eps)2c_*$.
\end{corollary}

\begin{proof}
  The result will follow from Theorem \ref{thm:perturb} by estimating
  the left hand side of \eqref{eq:ass-rates}. First, note that
  \begin{align*}
    \frac{q(x,\delta,\eta)}{q(x,\delta,\delta)}=\frac{c(x,\eta)}{c(x,\delta)}\leq \beta\;.
  \end{align*}
  Now, if $\delta,\eta$ are such that $c(\delta x,\eta)=c(x,\eta)$ and
  $c(\eta x,\delta)=c(x,\delta)$ for all $x$, then, using the detailed-balance condition, we infer that $q(x,\delta,\eta) =
  q_*(x,\delta,\eta)$. Otherwise, we have the bound
  $q(x,\delta,\eta)\leq e^{2\alpha}q_*(x,\delta,\eta)$. Note also that
  by construction we have that $q_*(x,\delta,\eta)=q_*(\delta
  x,\delta^{-1},\eta)$. This implies that
  \begin{align*}
   \frac{(q-q_*)(\delta x,\delta^{-1},\eta)}{q(x,\delta,\delta)}\leq \beta \frac{(q-q_*)(\delta x,\delta^{-1},\eta)}{q(x,\delta,\eta)}
   \leq \beta \frac{(q-q_*)(\delta x,\delta^{-1},\eta)}{q_*(x,\delta,\eta)}\leq \beta\Big(e^{2\alpha}-1\Big)\;.
  \end{align*}
  From this we obtain that
  \begin{align*}
    \lambda = \min_{x,\delta}c(x,\delta)\Big[1-\sum_{\eta\neq\delta}\frac{(q-q_*)(\delta x,\delta,\eta)}{q(x,\delta,\delta)}\Big]
            \geq c_*\Big[1-N\beta\Big(e^{2\alpha}-1\Big)\Big] = c_*(1-\eps)\;.
  \end{align*}
\end{proof}

\subsubsection{Conjugacy-invariant Cayley graphs}
\label{sec:non-commutativ}

Here we establish entropic Ricci bounds for a class of Markov chains
with not necessarily commutative mapping representation. Namely, we
consider random walks on weighted conjugacy-invariant Cayley graphs.

Let $\G$ be a finite group and let $G$ be a set of generators for
$\G$, i.e.~every $g\in\G$ can be written as a word
$g=\delta_1\delta_2\cdots\delta_n$ for suitable $\delta_i\in G$. We
assume that $G$ is 
\begin{itemize}
\item[(i)] closed under taking inverse: $\delta^{-1}\in G$ for all $\delta\in G$,
\item[(ii)] conjugacy-invariant: $g\delta g^{-1}\in G$ for all
  $\delta\in G,g\in\G$.
\end{itemize}

The Cayley graph associated to the generating set $G$ is the
(directed) graph with vertex set $\G$ and edge set $E=\{(x,y):
x,y\in\G,\ x^{-1}y\in G\}$. We consider a natural irreducible Markov
dynamics on the group $\G$ by choosing a function $c:\G\times G\to
(0,\infty)$ and considering the mapping representation $(G,c)$. The
associated Markov triple $(\G,Q,\pi)$ is the natural random walk on
the weighted directed graph $(\G,E)$, where $c(x,\delta)$ is
considered as the weight of the edge $(x,\delta x)$.

We have the following perturbative Ricci bound in the present situation:
\begin{theorem}\label{thm:Cayley-Ricci}
  Let us set
  \begin{align*}
\alpha_1 &:= \max\Big\{\log \frac{c(\delta x,\delta)}{c(x,\delta)} : x\in\G,\delta\in G \text{ with } \delta\neq\delta^{-1}, c(x,\delta)>0\Big\}\;,\\
    \alpha_2 &:= \max\Big\{\log \frac{c(\delta x,\eta)}{c(x,\eta)},\log \frac{c(\delta x,\delta\eta\delta^{-1})}{c(x,\eta)} : x\in\G,\delta,\eta\in G \text{ with } \delta\neq \eta,\eta^{-1}, c(x,\eta)>0\Big\}\;,\\
    \beta &:= \max\Big\{ \frac{c(x,\eta)}{c(x,\delta)} :
    x\in\G,\delta,\eta\in G \text{ with } c(x,\delta)>0\Big\}\;,
  \end{align*}
  and assume that
  $\eps:=e^{\alpha_1}+\beta(|G|-2)\Big(e^{2\alpha_2}-1\Big)\leq1$. Then we
  have 
  \begin{align*}
    \Ric(\G,Q,\pi)\geq (1-\eps)2c_*\;,
  \end{align*}
  where $c_*$ is the minimal transition rate defined in
  \eqref{eq:min-rate}.

  Moreover, if we assume that $\delta=\delta^{-1}$ for all $\delta\in
  G$ and that $\eps':=\beta(|G|-1)\Big(e^{2\alpha_2}-1\Big)\leq1$, we then
  have the improved bound $\Ric(\G,Q,\pi)\geq (1-\eps')2c_*$.
\end{theorem}
In particular, we obtain a Ricci bound for the simple random walk on a
conjugacy-invariant Cayley graph. In this case, we have for some
constant $c>0$ that
\begin{align*}
  c(x,\delta) = c \qquad\forall x\in G, \delta\in G\;.
\end{align*}
\begin{corollary}\label{cor:srw}
  The simple random walk on a conjugacy-invariant Cayley graph
  satisfies $\Ric(\G,Q,\pi)\geq 0$. If $\delta=\delta^{-1}$ holds for all
  $\delta\in G$, then we even have that $\Ric(\G,Q,\pi)\geq 2c$.
\end{corollary}

In Section~\ref{sec:cycles}, we will apply Corollary~\ref{cor:srw} to
analyze the curvature of some random walks on the symmetric group.

\medskip

Since the mapping representation in the present situation is not
commutative, the Bochner--Bakry--\'Emery method developed in
\cite{FM15} (see Prop. \ref{prop:BBE}) does not apply
immediately. Instead, we will combine it with a technique developed in
\cite{EMT15} which consists in partitioning the $\cB$-term into
contributions coming from square subgraphs of $(\G,E)$. We need some
notation before we come to the proof of Theorem
\ref{thm:Cayley-Ricci}.

A \emph{square} in the Cayley graph is a set
$\Box=\{x_1,x_2,x_3,x_4\}$ such that $x_{i+1}x_i^{-1}\in G$ for all
$i=1,\ldots,4$ with the convention that $x_5=x_1$. We write for short
$\delta_i:=x_{i+1}x_i^{-1}$. Given two maps $\delta$ and
$\eta\neq\delta,\delta^{-1}$ in $G$, we obtain for each $x\in\G$ a
square
\begin{align}\label{eq:box}
  \Box(x,\delta,\eta)=\{x_1=x,x_2=\delta x,x_3=\eta\delta x,x_4=\eta
  x\}\;.
\end{align}
Indeed, by invariance of $G$ under conjugation, we have that
$x_4x_3^{-1}=\eta x(\delta\eta x)^{-1}=\eta\delta^{-1}\eta^{-1}\in
G$. The other relations $x_{i+1}x_i^{-1}\in G$ for $i=1,2,3$ hold
trivially. The squares obtained in this way fall into two classes
depending on whether $\delta$ and $\eta$ commute or not. Let $S_1$ be
the collection of all squares obtained from commuting maps
$\delta,\eta$ and let $S_2$ denote the collection of all squares
obtained form non-commuting maps.

Given such a square $\Box$ and two functions
$\rho\in\cP_*(\G),\psi\in\R^\G$ we set
\begin{align*}
  \cB^{\mathrm{diag}}_\Box(\rho,\psi) &= \sum\limits_{i=1}^4 B(x_i,\delta_i,\delta_i)q(x_i,\delta_i,\delta_i) + B(x_i,\delta_{i-1}^{-1},\delta_{i-1}^{-1})q(x_i,\delta_{i-1}^{-1},\delta_{i-1}^{-1})\;,\\
  \cB^{\mathrm{off}}_\Box(\rho,\psi) &= \sum\limits_{i=1}^4 \Big[B(x_i,\delta_i,\delta_{i-1}^{-1}) + B(x_i,\delta_{i-1}^{-1},\delta_i)\Big]q(x_i,\delta_i,\delta_{i-1}^{-1})\;,\\
\end{align*}
as well as $\cB_\Box(\rho,\psi)=
\cB^{\mathrm{diag}}_\Box(\rho,\psi)+\cB^{\mathrm{off}}_\Box(\rho,\psi)$. Note
that $\cB_\Box(\rho,\psi)$ is the quantity $\cB$ calculated in the
square graph $\Box$ with the restrictions of $\rho,\psi$ to $\Box$.
We will proceed by rearranging the $\cB$-term of the full Cayley graph
into contributions from squares and apply the techniques of the
previous section separately in each square.

\begin{proof}[Proof of Thm.~\ref{thm:Cayley-Ricci}]
  We have to show that 
 $$\cB(\rho,\psi)\geq
  \left[1-e^{\alpha_1}-\beta(|G|-1)\Big(e^{2\alpha_2}-1\Big)\right]2c_*\cA(\rho,\psi)$$
  holds for any $\rho\in \cP_*(\G),\psi\in\R^\G$. We drop $\rho,\psi$
  from the notation for the rest of the proof. We distinguish on- and
  off-diagonal contributions to $\cB$ by writing
  $\cB=\cB^{\mathrm{diag}}+\cB^{\mathrm{off,1}}+\cB^{\mathrm{off,2}}$ with
  \begin{align*}
    \cB^{\mathrm{diag}}&= \sum\limits_{x\in \G,\delta\in G}B(x,\delta,\delta)q(x,\delta,\delta)\;,\\
    \cB^{\mathrm{off,1}}&= \sum\limits_{x\in \G,\delta\in G:\delta\neq\delta^{-1}}B(x,\delta,\delta^{-1})q(x,\delta,\delta^{-1})\;,\\
    \cB^{\mathrm{off,2}}&= \sum\limits_{x\in \G,\delta,\eta\in G: \eta\neq\delta,\delta^{-1}}B(x,\delta,\eta)q(x,\delta,\eta)\;.
  \end{align*}
  We first estimate $ \cB^{\mathrm{off,1}}$. Symmetrizing in
  $\delta,\delta^{-1}$ and using Lemma \ref{lem:pointwise}, we obtain
  \begin{align}\label{eq:Boff-rewrite0}
     \cB^{\mathrm{off,1}} \geq - \sum\limits_{x\in \G,\delta\in G:\delta\neq\delta^{-1}}B(x,\delta,\delta)q(\delta x,\delta^{-1},\delta)
     \geq -e^{\alpha_1} \cB^{\mathrm{diag}}\;.
  \end{align}
 Now, we claim that
  \begin{align}\label{eq:Boff-rewrite1}
    \cB^{\mathrm{off,2}}&= \sum\limits_{\Box\in S_1}\cB^{\mathrm{off}}_\Box +\frac12\sum\limits_{\Box\in S_2}\cB^{\mathrm{off}}_\Box\;,\\\label{eq:Boff-rewrite2}
\cB^{\mathrm{diag}}&= \frac{1}{|G|-1}\left[\sum\limits_{\Box\in S_1}\cB^{\mathrm{diag}}_\Box +\frac12\sum\limits_{\Box\in S_2}\cB^{\mathrm{diag}}_\Box\right]\;.
  \end{align}
  Indeed, each term $B(x,\delta,\eta)q(x,\delta,\eta)$ appears in
  exactly one square from $S_1$, namely the square $\Box(x,\delta,\eta)$
  defined in \eqref{eq:box}, if $\delta$ and $\eta$ commute. If they
  do not commute, then the term $B(x,\delta,\eta)q(x,\delta,\eta)$
  appears in exactly two squares from $S_2$, namely
  $\Box(x,\delta,\eta)$ and $\Box(x,\eta,\delta)$. Moreover, each term
  $B(x,\delta,\delta)q(x,\delta,\delta)$ appears in exactly $N_1$
  squares in $S_1$ and in exactly $2N_2$ squares in $S_2$ with
  $N_1=\#\{\eta\in G:\eta\neq \delta, \delta^{-1},\delta\eta=\eta\delta\}$ and
  $N_2=\#\{\eta\in \G: \delta\eta\neq \eta\delta\}$. Obviously,
  $N_1+N_2\leq|G|-1$.
  
  To calculate the $\cB$-terms in each square $\Box$, we apply the
  techniques of the previous section by choose a new mapping
  representation consisting of two maps $\delta,\tilde\delta$ that is
  involutive and commutative. For instance set $\delta x_i= x_2, x_1,
  x_4, x_3\;, \tilde \delta x_i= x_4, x_3, x_2, x_1$ for $i=1,2,3,4$.
  Thus, following the proofs of Theorem \ref{thm:perturb} and
  Corollary \ref{cor:perturb}, we find
  \begin{align*}
    \cB^{\mathrm{off}}_\Box \geq - \beta\Big(e^{2\alpha_2}-1\Big)\cB^{\mathrm{diag}}_\Box\;.
  \end{align*}
  Combing this with \eqref{eq:Boff-rewrite0}, \eqref{eq:Boff-rewrite1}, \eqref{eq:Boff-rewrite2}
  and using Lemma \ref{lem:B-diagonal} yields
 \begin{align*}
   \cB&=\cB^{\mathrm{diag}}+\cB^{\mathrm{off,1}}+\cB^{\mathrm{off,2}}
   \geq \cB^{\mathrm{diag}}\left[1-e^{\alpha_1}-(|G|-1)\beta\Big(e^{2\alpha_2}-1\Big)\right]\\
 &\geq \left[1-e^{\alpha_1}-(|G|-1)\beta\Big(e^{2\alpha_2}-1\Big)\right]2c_*\cA\;,
 \end{align*}
 which finishes the proof of the first statement. To obtain the second
 statement, we simply note that $\cB^{\mathrm{off,1}}=0$ if
 $\delta=\delta^{-1}$ for all $\delta\in G$.
\end{proof}

\section{Examples}
\label{sec:examples}

In this section we apply our perturbation method to derive Ricci
bounds in concrete models. First, we consider a general Ising model in
the high temperature regime. Then, we specialize this result to obtain
bounds for the Ising model on a finite sub-lattice of $\Z^d$ and the
Curie--Weiss model. Moreover, we consider a general hard-core model
put forward in \cite{DPP} and extend the results on convex entropy
decay obtained there to the level of Ricci curvature.

\subsection{Bounds for a general Ising model}
\label{sec:Ising-general}

Let $n\in\N$ and introduce the state space $\cX=\{-1,1\}^n$. Let
$k\in\R^{n\times n}$ be a matrix modeling the interaction strength
between the sites. We set $k_{ii}=0$ for all $i$. Then we introduce
the Hamiltonian $H:\cX\to\R$ via
\begin{align*}
  H(x)~=~ - \sum\limits_{i,j=1}^n k_{ij}x_ix_j\;.
\end{align*}
Note that we make no assumption on the sign of $k$. We consider the probability measure
\begin{align*}
  \pi_\beta(x)~=~\frac1{Z_\beta}\exp\big(-\beta H(x)\big)\;,
\end{align*}
where $Z_\beta$ is a normalizing constant and $\beta\in(0,\infty)$
denotes the inverse temperature. We consider the associated Glauber
dynamics, the continuous time Markov chain given by the q-matrix
\begin{align*}
  Q_\beta(x,y)
   ~=~
   \begin{cases}
     \sqrt{\frac{\pi_\beta(y)}{\pi_\beta(x)}}\;, & \text{if } \norm{x-y}_{l^1}=1\;,\\
     0\;, & \text{else.}
   \end{cases}
\end{align*}
A natural mapping representation is given as follows. Let
$G=\{\delta_i,~i=1,\dots,n\}$, where $\delta_i:\cX\to\cX$ is the map
flipping the $i$-th coordinate, i.e. $\big(\delta_i(x)\big)_i=-x_i$
and $\big(\delta_i(x)\big)_j=x_j$ for all $j\neq i$. Then we put
\begin{align*}
  c(x,\delta_i)~=~\sqrt{\frac{\pi_\beta(\delta_i x)}{\pi_\beta(x)}}~=~e^{-\frac{\beta}{2}\nabla_i H(x)}\;.
\end{align*}
where we write for short $\nabla_i H(x)=\nabla_{\delta_i}
H(x)=H(\delta_ix)-H(x)$.  Note that this mapping representation is
commutative and involutive, i.e.~$\delta_i^{-1}=\delta_i$. 

We have the following Ricci bound for the Glauber dynamics of the
general Ising model.

\begin{theorem}\label{thm:Ric-Ising}
  Assume that 
  \begin{align}\label{eq:eps-beta}
    \eps(\beta) ~:=~ \max_i\sum\limits_{j,j \neq
      i}\exp\left(2\beta\sum\limits_{m\neq
        i,j}|k_{im}|+|k_{jm}|\right)\Big(e^{4\beta |k_{ij}|}-1\Big)
    ~\leq~1\;.
  \end{align}
  Then the Glauber dynamics satisfies
  \begin{align*}
    \Ric(\cX,Q_\beta,\pi_\beta)~\geq~\big(1-\eps(\beta)\big)2c_*\;,
  \end{align*}
  where $c_*=\min\{c(x,\delta): x,\delta\}$
  denotes the minimal transition rate.
\end{theorem}

\begin{proof}
  The claim is a consequence of the first part of Theorem
  \ref{thm:perturb} once we have established the following
  estimate. Let $q_*$ be defined as in \eqref{eq:qstar1}. Then, for
  all $x\in\cX$ and all $i,j=1,\dots, n$ we have:
  \begin{align}\label{eq:q-Ising}
    \frac{q(\delta_i x,\delta_i,\delta_j)-q_*(x,\delta_i,\delta_j)}{q(x,\delta_i,\delta_i)}
   ~\leq~
   \exp\left(2\beta\sum\limits_{m\neq i,j}|k_{im}|+|k_{jm}|\right)\Big(e^{4\beta |k_{ij}|}-1\Big)\;.
 \end{align}
 Indeed, we first note that
  \begin{align*}
    q(x,\delta_i,\delta_j)~=~\exp\left(-\frac{\beta}{2}\big(H(\delta_ix) + H(\delta_jx)\big)\right) \;.
  \end{align*}
  Note further that for $i\neq j$ we have
  \begin{align*}
    H(x)~&=~-\sum\limits_{l,m\neq i,j}k_{lm}x_l x_m - 2\sum\limits_{m\neq i,j}k_{im}x_i x_m- 2\sum\limits_{m\neq i,j}k_{jm}x_j x_m -2k_{ij}x_i x_j\;,\\
    H(\delta_i x)~&=~-\sum\limits_{l,m\neq i,j}k_{lm}x_l x_m + 2\sum\limits_{m\neq i,j}k_{im}x_i x_m -2\sum\limits_{m\neq i,j}k_{jm}x_j x_m +2k_{ij}x_i x_j\;,
   \end{align*}
  which yields
  \begin{align*}
    H(\delta_i x) + H(\delta_j x) ~=~ -2\sum\limits_{l,m\neq i,j}k_{lm}x_l x_m + 4 k_{ij}x_i x_j\;.  
  \end{align*}
  Since the first term does not depend on the coordinates $i,j$, we get for $y\in\{x,\delta_ix,\delta_j x,\delta_i\delta_jx\}$:
  \begin{align*}
    q(y,\delta_i,\delta_j) ~=~ \exp\left(\beta\sum\limits_{l,m\neq i,j}k_{lm}x_lx_m\right)\exp\big(-2\beta k_{ij} y_i y_j\big)\;,
  \end{align*}
  and we conclude that
 \begin{align*}
   q(\delta_ix,\delta_i,\delta_j)-q_*(x,\delta_i,\delta_j)~\leq~ \exp\left(\beta\sum\limits_{l,m\neq
       i,j}k_{lm}x_lx_m\right)\Big(e^{2\beta |k_{ij}|}-e^{-2\beta
     |k_{ij}|}\Big)\;.
 \end{align*}
 Similarly, noting that $q(x,\delta_i,\delta_i)=\exp\big(-\beta H(\delta_ix)\big)$,
 we obtain the estimate
  \begin{align}\nonumber
q(x,\delta_i,\delta_i)
   ~&=~\exp\left(\beta\sum\limits_{l,m\neq
       i,j}k_{lm}x_lx_m\right)\exp\left(-2\beta \left[\sum\limits_{m\neq i,j}(k_{im}x_i-k_{jm}x_j)x_m+k_{ij}x_ix_j\right]\right)\\
\label{eq:frust}
    ~&\leq~\exp\left(\beta\sum\limits_{l,m\neq
       i,j}k_{lm}x_lx_m\right)
    \exp\left(-2\beta\left(|k_{ij}|+\sum\limits_{m\neq i,j}|k_{im}|+|k_{jm}|\right)\right)\;,
  \end{align}
  which yields the claim \eqref{eq:q-Ising}. Thus, by
  \eqref{eq:eps-beta} we find that
  \begin{align*}
    \lambda = \min_{x,i}c(x,\delta_i)\Big[1-\sum_{j\neq i} \frac{(q-q_*)(\delta_i x,\delta_i,\delta_j)}{q(x,\delta_i,\delta_i)}\Big]
    \geq c_*\Big(1-\eps(\beta)\Big)\geq0\;.
  \end{align*}
  Hence, the assumption \eqref{eq:ass-rates} of Theorem
  \ref{thm:perturb} is satisfied and the thesis follows.
\end{proof}

\begin{remark}\label{rem:ferro}
  In \eqref{eq:q-Ising} we have given a worst-case estimate in
  terms of the absolute value of the interaction. This estimate seems
  rather sharp if the interaction is ferromagnetic, i.e.~$k_{ij}\geq0$
  for all $i,j$. However, in models where the interaction matrix
  changes sign, a finer estimate making use of frustration effects
  should be possible. This concerns the second line of
  \eqref{eq:frust}.
\end{remark}

Let us now specialize our result to the $d$-dimensional Ising model
and the Curie--Weiss model.

\subsubsection{The d-dimensional Ising model}
\label{sec:d-dimIsing}

Let $\Lambda$ be a finite connected subset of $\Z^d$ endowed with the
natural graph structure. Put $n=\abs{\Lambda}$. We consider the
Hamiltonian $H:\{-1,1\}^\Lambda\to\R$ given by
\begin{align*}
  H(x)~=~-\frac12\sum\limits_{i\sim j}x_ix_j\;,
\end{align*}
where $i\sim j$ means that $i$ and $j$ are adjacent. Note that this is
of the form that we considered in the previous section, namely it
corresponds to choosing the matrix $k\in\R^{n\times n}$ as
\begin{align*}
  k_{ij}~=~
  \begin{cases}
    \frac12\;, & i\sim j\;,\\
     0 \;, & \text{else.}
  \end{cases}
\end{align*}
Noting that each site has at most $2d$ neighbors we have the following bound
\begin{align}\label{eq:eps-beta-dIsing}
  \eps(\beta)~\leq~(2d-1)e^{2\beta(2d-1)}\big(e^{2\beta}-1\big)\;.
\end{align}
Moreover, the minimal transition rate for the Glauber dynamics becomes $c_*=e^{-\beta d}$.

\begin{corollary}\label{cor:d-Ising}
  Assume that $\eps(\beta)\leq 1$. Then the Glauber dynamics for the
  $d$-dimensional Ising model satisfies
  \begin{align*}
    \Ric(Q_{Gl})~\geq~\big(1-\eps(\beta)\big)2e^{-\beta d}\;,
  \end{align*}
  where $\eps(\beta)$ is given by \eqref{eq:eps-beta-dIsing}.
\end{corollary}

In particular, for $d=2$ we see using the bound
\eqref{eq:eps-beta-dIsing} that the condition $\eps(\beta)\leq1$ is
satisfied if
\begin{align*}
  3e^{6\beta}\big(e^{2\beta}-1\big)~\leq~1\;,\quad\text{resulting in approximately}\quad \beta~\leq~0.089\;.
\end{align*}

\subsubsection{The Curie--Weiss model}
\label{sec:CW}

We consider the Hamiltonian $H:\{-1,1\}^n\to\R$ given by
\begin{align*}
  H(x)~=~-\frac1{2n}\sum\limits_{i,j=1}^nx_ix_j\;,
\end{align*}
Note that this is of the form we considered in the previous
section -- corresponds to choosing the matrix $k\in\R^{n\times
  n}$ as
\begin{align*}
  k_{ij}~=~\frac1{2n}\quad\forall i,j=1,\dots,n\;.
\end{align*}
Thus we see that \eqref{eq:eps-beta} turns into
\begin{align}\label{eq:eps-beta-CW}
  \eps(\beta)~=~(n-1) e^{2\beta\frac{n-2}{n}}\big(e^{2\beta\frac{1}{n}}-1\big)\;.
\end{align}
Moreover, the minimal transition rate for the Glauber dynamics becomes
$c_*=e^{-\beta\frac{n-1}{2n}}$.

\begin{corollary}\label{cor:CW}
  Assume that $\eps(\beta)\leq 1$. Then the Glauber dynamics for the
  Curie--Weiss model satisfies
  \begin{align*}
    \Ric(Q_{Gl})~\geq~\big(1-\eps(\beta)\big)2e^{-\beta\frac{n-1}{2n}}\;,
  \end{align*}
  where $\eps(\beta)$ is given by \eqref{eq:eps-beta-CW}.
\end{corollary}

Note that if we disregard corrections of the order $O(\frac1n)$ the
condition $\eps(\beta)\leq1$ corresponds to
\begin{align*}
   2\beta e^{2\beta}~\leq~1\;,\quad\text{approximately}\quad\beta~\leq~ 0.284\;.
\end{align*}

Recall from Section \ref{sec:prelim} that an entropic Ricci bound
$\Ric\geq \kappa>0$ implies the modified logarithmic Sobolev
inequality \eqref{eq:MLSI}. Thus, we obtain in particular that the
Glauber dynamics for the Curie--Weiss model satisfies MLSI up to the
inverse temperature $\beta\approx 0.284$. In a recent preprint, Marton
\cite{Mar15} showed that the MLSI holds up to the critical inverse
temperature $\beta=1$, which is beyond the scope of our perturbative
approach. It remains an open question to determine the optimal Ricci
bound for the Curie--Weiss model.

\subsection{Bounds for a general hard-core model}

In this section we derive Ricci bounds for a general hard-core model
put forward in \cite{DPP}.

Let $T$ be a finite set and consider the configuration space $S :=
\{x: T \to \N\cup\{0\}\}$. A set $A\subset
S$ is called {\em decreasing} if for all $x,y\in S$, it holds
\begin{align*}
  x\in A\;,\ y(i)\leq x(i), \ \  \forall i\in T \quad\Rightarrow\quad y\in A\;.
\end{align*}
We fix a finite decreasing set $A$ and call it the set of {\em allowed
  configurations}. We fix a function $\nu: T \to (0,\infty)$, called
the {\em intensity}, and define a probability measure $\pi$ on
$\cX:=A$ by
\begin{align*}
  \pi(x) = \frac{1}{Z} \prod_{i\in T} \frac{\nu(i)^{x(i)}}{x(i)!},
\end{align*}
where $Z$ is a normalization constant. A Markov dynamics on $\cX$ is
given by the rate matrix
\begin{align*}
  Q(x,y) :=
  \begin{cases}
    \nu(i)\one_{\{x+\one_i\in A\}}\;, & \text{if } y=x+\one_i\;,\\
    x(i)\one_{\{x-\one_i\in A\}}\;, & \text{if } y=x-\one_i\;,\\
    0\;, & \text{else}\;.
  \end{cases}
\end{align*}
Note that this dynamics is reversible w.r.t.~$\pi$. A natural mapping
representation for this model is given on the extended state space
$\cX'=S$ (c.f. Remark \ref{rmk:enlarge}) as follows. Let
\begin{align*}
  G = \{\gamma_i^+, \gamma_i^-: i \in T\}\,,
\end{align*}
where $\gamma_i^+,\gamma_i^-:S\to S$ are the creation and annihilation
maps defined by
\begin{align*}
  \gamma_i^+(x) =  x + \one_i\;,\qquad  \gamma_i^-(x) = \begin{cases}
    x - \one_i\;, &\text{if } x(i)>0\;,\\
    x\;, &\text{else }.
  \end{cases}
\end{align*}
We then may define the transition rates $c:\cX'\times G\to\R_+$ by
\begin{align*}
  c(x, \gamma_i^+) = \nu(i) \one_{\{x + \one_i \in A\}}\;,\qquad c(x, \gamma_i^-) = x(i) \one_{\{x(i)>0\}}\;.
\end{align*}

Define
\begin{align}\label{eq:epsilon-def1}
	\epsilon_0 &= \max_{x\in A, i\in T: x(i)>0}  \sum_{j \neq i} \nu(j) \one_{\{x + \one_j - \one_i \in A\}} \one_{\{x + \one_j \notin A\}}\;,\\
\label{eq:epsilon-def2}
	\epsilon_1 &= \min_{x\in A, i\in T: x(i)>0} \nu(i) \one_{\{x+\one_i\notin A\}}\;.
\end{align}

We have the following entropic Ricci curvature bound for the general
hard-core model.

\begin{theorem}\label{thm:hard-core}
  Assume that $\epsilon_0\leq 1$. Then, we have that
  \begin{align*}
    \Ric(\cX,Q,\pi) \geq \frac12(1 - \epsilon_0 + \epsilon_1)\;.
  \end{align*}
\end{theorem}

\begin{proof}
  The result will be a consequence of the second part of Theorem
  \ref{thm:perturb}.  We let $H_1=\{\gamma^+_i:i\in T\}$ and
  $H_2=\{\gamma^-_i:i\in T\}$. One readily checks that for all $x\in
  S$ and all $i\neq j$:
  \begin{align*}
  0=(q-q_*)(x,\gamma_i^-,\gamma_j^-)=(q-q_*)(x,\gamma_i^-,\gamma_j^+)=(q-q_*)(x,\gamma_i^{+},\gamma_j^{-})\;.
  \end{align*}
  Moreover, we have
  \begin{align*}
    (q-q_*)(x,\gamma_i^+,\gamma_j^+)=\one_{\{x+\one_i\in
    A\}}\one_{\{x+\one_j\in A\}}\one_{\{x+\one_i+\one_j\notin
    A\}}\nu(i)\nu(j)\pi(x)\;.
  \end{align*}
  This yields
  \begin{align*}
    \frac{(q-q_*)(\delta x,\delta^{-1},\eta)}{c(x,\delta)\pi(x)}
    =
    \begin{cases}
      0\;, & \delta=\gamma_i^+, \eta=\gamma_j^-\;,\\
      0\;, & \delta=\gamma_i^+, \eta=\gamma_j^+\;,\\
      0\;, & \delta=\gamma_i^-, \eta=\gamma_j^-\;,\\
      \one_{\{x(i)>0\}}\one_{\{x-\one_i+\one_j\in A\}}\one_{\{x+\one_j\notin
    A\}}\nu(j)\;, & \delta=\gamma_i^-, \eta=\gamma_j^+\;.\\
    \end{cases}
  \end{align*}
In the notation of Theorem \ref{thm:perturb} we thus obtain
\begin{align*}
  \lambda_1 &= \min_{x\in S,i\in T} c(x,\gamma_i^+)-c(\gamma_i^+x,\gamma_i^+) 
            = \min_{x\in S,i\in T} \nu(i)\one_{\{x+\one_i\in A\}}\one_{\{x+2\cdot\one_i\notin A\}}\\
            &= \min_{x\in A,i\in T:x(i)>0} \nu(i)\one_{\{x+\one_i\notin A\}}=\eps_1\;.
\end{align*}
Moreover, we get
\begin{align*}
  \lambda_2 &= \min_{x\in S,i\in T} c(x,\gamma_i^-)-c(\gamma_i^-x,\gamma_i^-)-\sum_{j\neq i}\nu(j) \one_{\{x(i)>0\}}\one_{\{x-\one_i+\one_j\in A\}}\one_{\{x+\one_j\notin A\}}\\
           &=  \min_{x\in A,i\in T:x(i)>0} x(i) - \big(x(i)-1\big) - \sum_{j\neq i}\nu(j) \one_{\{x-\one_i+\one_j\in A\}}\one_{\{x+\one_j\notin A\}}
           = 1-\eps_0\;.
\end{align*}
Applying the second part of Theorem \ref{thm:perturb} yields the
claim.
\end{proof}

\begin{remark}
  Under the same assumptions as in Theorem \ref{thm:hard-core}, Dai Pra
  and Posta established in \cite{DPP} the convex entropy decay
  inequality \eqref{eq:ced} with $\kappa=1-\eps_0+\eps_1$. Recall from
  Section \ref{sec:prelim} that $\Ric\geq \kappa$ implies
  \eqref{eq:ced}. Thus, by the previous theorem we recover, in
  particular, the result in \cite{DPP} up to a factor $1/2$ in the
  constant. Note also that the choice of the admissible function $R$
  implicit in the use of Theorem \ref{thm:perturb} coincides with the
  choice made in \cite{DPP}.
\end{remark}

Let us specialize our result to the standard hard-core model and a
model for long hard rods.

\subsubsection{The hard-core model}

Let $G = (V,E)$ be a finite, connected graph without self-loops.
Using our notation above, we let $T := V$, $\nu(i) \equiv \rho$ for
some constant $1>\rho>0$, and
\[
	A := \{x \in S: x(i) \in \{0,1\} \text{ for all } i \in T, x(i)y(j) = 0 \text{ for all } \{i,j\}\in E\}.
\]
We define $\Delta$ to be the maximum degree of any vertex in the
graph.

It is easy to see that in this case
\eqref{eq:epsilon-def1} and \eqref{eq:epsilon-def2} become
$\epsilon_0= \rho\Delta$ and $\eps_1=\rho$. Thus, we obtain the
following corollary.

\begin{corollary}
  If $\rho \leq 1/\Delta$, then we have
  $\Ric(\cX,Q,\pi)\geq\frac12\big(1-\rho(\Delta-1)\big)$.
\end{corollary}

This model has been widely studied in the literature.  We refer the
interested reader to the book of Levin, Peres and Wilmer \cite{LPW09}
and the works of Luby and Vigoda \cite{LV99} and Vigoda \cite{Vig01}
concerning fast mixing results for the hard-core model.

\subsubsection{Long hard rods}

Fix two natural numbers $L$ and $k$ in the regime where $L \gg k$.
Define the space $T_-$ of horizontal rods of length $k$ to be the
collection of all sequence of adjacent vertices in $\{0,1,\dots,L\}^2$
of the form
\[
	\{(u_1, u_2), (u_1+1, u_2), \dots, (u_1 + k, u_2)\}.
\]
Similarly, define the space $T_+$ of vertical rods of length $k$ to be
the collection of all sequence of adjacent vertices in
$\{0,1,\dots,L\}^2$ of the form
\[
	\{(u_1, u_2), (u_1, u_2+1), \dots, (u_1, u_2+k)\}.
\]
We then define $T$, the set of all rods, as the union of $T_-$ and
$T_+$. The admissible set is defined to be
\[
	A = \{x \in S : x(i) \in \{0,1\} \text{ for all } i\in T, x(i) y(j) = 0 \text{ if } i \neq j \text{ and } i\cap j \neq \emptyset\}.
\]
In other words, we wish to only allow rods which do not touch.

Further, we let $\nu(i) \equiv \rho$ for some constant $\rho > 0$. It
is easy to check that in this case $\epsilon_0 = \rho (k^2 + 4k + 1)$
and $\epsilon_1 = \rho$. Thus we obtain the following corollary.

\begin{corollary}
  If $\rho \leq 1/( k^2 + 4k + 1)$, we have $\Ric(\cX,Q,\pi)\geq
  \frac12\Big(1-\rho(k^2+4k)\Big)$.
\end{corollary}

As pointed out by Disertori and Giuliani \cite{DG13}, for $k$
sufficiently large, there is a phase transition as $L$ tends to
infinity at some critical value $\rho_c$, which is expected to be of
the order $k^{-2}$. Just as for the convex entropy decay considered
in \cite{DPP}, our work above yields a uniform (in $L$) curvature
bound in the asymptotically correct regime.

\subsection{Random walks on the symmetric group}
\label{sec:cycles}

Let us briefly highlight a class of examples where Corollary \ref{cor:srw} applies.

Consider the symmetric group $S_n$ of all permutations on $n$
letters. A conjugacy-invariant set of generators is given for instance
by the set $G_{n,k}$ of all $k$-cycles in $S_n$ for $1<k<n$. Here a
$k$-cycle is a cyclic permutation of length $k$. The simple random walk
$Q_{n,k}$ on the associated Cayley graph is given by
$c_{n,k}(x,\delta)\equiv |G_{n,k}|^{-1}=\binom{n}{k}^{-1}$. It is reversible
w.r.t.~the uniform probability measure $\pi_n$ on $S_n$. 

\begin{corollary}
  The simple random walk on $S_n$ generated by $k$-cycles satisfies
  $$\Ric(S_n,Q_{n,k},\pi_n)\geq 2\binom{n}{k}^{-1}\;.$$
\end{corollary}

In the case of $2$-cycles or transpositions, we recover the result
obtained in \cite[Thm.~1.2]{EMT15}. In this case, the optimal constant
$\kappa$ in the MLSI \eqref{eq:MLSI} is known to satisfy the bounds
$1/2(n-1)\leq\kappa\leq 2/(n-1)$. Thus the Ricci bound that we obtain
differs roughly by a factor $n$. It is an open question to determine
the correct order for the Ricci bound.

\bibliographystyle{plain}
\bibliography{ricci}

\end{document}